\documentclass[11pt]{amsart}
\usepackage[T1]{fontenc}
\usepackage[utf8]{inputenc}
\usepackage{lmodern}
\usepackage{amsmath,amssymb,mathtools,mathrsfs}
\usepackage[margin=1.05in]{geometry}
\usepackage{microtype}
\usepackage[colorlinks=true,linkcolor=blue,citecolor=blue,urlcolor=blue]{hyperref}
\hypersetup{pdftitle={Morse index, topology, and ends of minimal surfaces with noncompact free boundary},pdfauthor={Marcio Batista and Matheus B. Martins}}
\numberwithin{equation}{section}
\newtheorem{theorem}{Theorem}[section]
\newtheorem{proposition}[theorem]{Proposition}
\newtheorem{lemma}[theorem]{Lemma}
\newtheorem{corollary}[theorem]{Corollary}
\newtheorem*{theoremA}{Theorem A}
\newtheorem*{theoremB}{Theorem B}
\newtheorem*{theoremC}{Theorem C}
\theoremstyle{definition}

\newtheorem{example}[theorem]{Example}
\theoremstyle{remark}
\newtheorem{remark}[theorem]{Remark}
\newcommand{\R}{\mathbb R}
\newcommand{\C}{\mathbb C}
\newcommand{\HH}{\mathcal H}
\newcommand{\EE}{\mathcal E}
\newcommand{\KK}{\mathcal K}
\newcommand{\NN}{\mathcal N}
\newcommand{\Sigmabar}{\overline\Sigma}
\newcommand{\dd}{\,d}
\DeclareMathOperator{\Ind}{Ind}
\DeclareMathOperator{\diver}{div}
\DeclareMathOperator{\supp}{supp}
\DeclareMathOperator{\tr}{tr}

\DeclareMathOperator{\spanop}{span}
\DeclareMathOperator{\Rea}{Re}
\DeclareMathOperator{\Ima}{Im}

\title[Morse index and noncompact free boundary]{Morse index, topology, and ends of minimal surfaces with noncompact free boundary}
\author{M\'arcio Batista}
\address{CPMAT--Instituto de Matemática, Universidade Federal de Alagoas, Macei\'o, AL, 57072-970, Brazil}
\email{mhbs@mat.ufal.br}
\author{Matheus B. Martins}
\address{Departamento de Matemática, Universidade Estadual de Alagoas, Arapiraca, 
AL, 57.312-270, Brazil}
\email{matheus.martins@uneal.edu.br}
\date{\today}
\begin{document}
\begin{abstract}
We establish index estimates for complete two-sided free boundary minimal surfaces in smooth mean-convex domains of $\R^3$ with noncompact boundary. We first prove $3\Ind_s(\Sigma)\ge 2g+b-1$, where $g$ and $b$ describe the conformal compactification. We then include all interior ends with their multiplicities, without further asymptotic assumptions, and selected boundary ends under an explicit condition ensuring vanishing cutoff errors. The proof combines a localized energy identity with a Riemann--Roch count on the conformal double. A boundary puncture at which the chosen forms are regular eliminates the exceptional space; if poles are allowed at every boundary puncture, this space has dimension at most one. We provide the mixed cutoff construction, examples distinguishing embedded boundary ends from reflected planar ends, and a separate analysis of the conformal Jacobi metric. The latter yields finite Dirichlet energy of the logarithmic conformal factor, but the unrestricted boundary-end estimate remains an open step. The compact-boundary case was treated by Cavalcante, Mendes, and dos Santos.
\end{abstract}
\maketitle
\enlargethispage{2pt}

\section{Introduction}\label{sec:introduction}
The Morse index measures the number of independent compactly supported variations that decrease the area of a minimal surface to second order. A central problem in geometric analysis is to relate this variational invariant to topology. Harmonic one-forms provide a particularly effective link: their dimension records topological information, while their Euclidean components yield scalar test functions for the Jacobi quadratic form. The difficulty is to turn the resulting energy identities into strictly negative directions, especially on a complete noncompact surface where the test functions need not have compact support.

For complete minimal surfaces without boundary, Fischer-Colbrie \cite{FC85} established a fundamental connection between finite index and conformal structure. Ros \cite{Ros06} developed the use of harmonic one-forms together
with a rigidity analysis of the forms whose components become Jacobi
functions. In a related direction, Savo \cite{Savo10} compared the
spectrum of the stability operator with that of the Hodge Laplacian
on one-forms for closed minimal hypersurfaces of the sphere.
Chodosh and Maximo \cite{CM16,CM23} combined the harmonic-form
approach with weighted spaces and carefully chosen cutoff functions
to obtain quantitative estimates involving genus and ends; their
later work also accounts for end multiplicity. These results show both the effectiveness of the harmonic-form method and the importance of controlling its equality case. An estimate involving the index and the nullity is often accessible before one understands which harmonic forms can actually produce null directions.

In the free boundary setting, the second variation contains a Robin boundary term determined by the second fundamental form of the supporting hypersurface. Sargent \cite{Sargent17} and Ambrozio, Carlotto, and Sharp \cite{ACS18} obtained topological index estimates for compact free boundary minimal hypersurfaces in Euclidean domains. Related arguments for constant mean curvature surfaces were developed by Cavalcante and de Oliveira \cite{CO20} and Aiex and Hong \cite{AH21}. Hong and Saturnino \cite{HS23} treated capillary surfaces, including finite-index conformal compactification in the noncompact case. Their Theorem~1.4 is especially relevant here: under the curvature assumptions satisfied by a minimal surface in a mean-convex Euclidean domain, the compactification may have punctures on its boundary as well as in its interior.

\emph{The case of complete noncompact free boundary minimal surfaces with compact boundary is treated by Cavalcante, Mendes, and dos Santos \cite{CMS26}.} Their compact-boundary setting involves removal of finitely many interior points from a compact bordered surface. This compact-boundary contribution is also described in the introduction of Mendes \cite{Mendes26}. Our concern is the additional case in which the geometric boundary itself is noncompact.

A further relevant observation comes from Mendes \cite{Mendes26}, who studies a planar Robin problem with compact boundary. His equality analysis uses the fact that a tangential harmonic vector field satisfying the vector Robin condition has constant tangential component along each boundary component. In the present setting, this conclusion follows intrinsically from the divergence equation and therefore applies also to nonflat surfaces. A boundary end then supplies an additional obstruction: a smooth one-form on the conformal compactification has finite absolute integral along an arc approaching a puncture, whereas a nonzero constant multiple of intrinsic arclength has infinite integral along a complete escaping boundary arc.

We state the result using the sign convention
\[
h^{\partial\Omega}(U,V)=\langle D_U\nu,V\rangle,
\qquad H_{\partial\Omega}=\tr h^{\partial\Omega},
\]
where $\nu$ is the outward unit normal. Thus a Euclidean sphere has positive mean curvature with this convention. The index $\Ind_s$ refers to all compactly supported scalar variations, with no volume constraint.

\begin{theoremA}\label{thm:main}
Let $\Omega\subset\R^3$ be a smooth mean-convex domain, and let
$F:\Sigma\to\overline\Omega$ be a smooth, connected, two-sided free boundary minimal immersion. Assume that the induced metric is complete as a metric space including $\partial\Sigma$, that $\partial\Sigma$ is noncompact, and that $\Ind_s(\Sigma)<\infty$. Then $\Sigma$ is conformally equivalent to
\(
\Sigmabar\setminus(P_{\mathrm{int}}\cup P_{\partial}),
\)
where $\Sigmabar$ is a compact connected bordered Riemann surface and the two puncture sets are finite, with $P_{\partial}\ne\varnothing$. If $g$ is the genus of $\Sigmabar$ and $b$ is the number of components of $\partial\Sigmabar$, then
\begin{equation}\label{eq:main}
\Ind_s(\Sigma)\ge
\left\lceil\frac{2g+b-1}{3}\right\rceil.
\end{equation}
In particular,
\[
\Ind_w(\Sigma)\ge
\left\lceil\frac{2g+b-1}{3}\right\rceil-1.
\]
\end{theoremA}

Here the free boundary condition includes
$F(\Sigma)\cap\partial\Omega=F(\partial\Sigma)$ and orthogonal intersection with the support. The number $b$ counts boundary circles of the compactification, not the connected components of the geometric boundary after punctures have been removed. 

There are two distinct difficulties associated with noncompact boundary. First, integration by parts at infinity must retain the Robin term without imposing a global trace estimate on the original one-form. We resolve this through an exact identity with compactly supported cutoffs. Second, the conformal compactification has boundary punctures, so the dimension of the harmonic-form space cannot be obtained by filling only interior ends. We prove removal of both types of punctures in detail. After these steps, noncompactness of the boundary becomes useful: together with completeness it eliminates the exceptional harmonic forms and removes the nullity term from the index estimate.

Theorem~A uses only square-integrable forms, which extend across all punctures. To detect ends, we allow prescribed poles. Write
\[
P_{\mathrm{int}}=\{p_1,\ldots,p_r\},\qquad
P_{\partial}=\{q_1,\ldots,q_s\},\qquad s\ge1.
\]
The finite-curvature conclusion of Hong--Saturnino implies that each interior end has the classical integer multiplicity $d_i\ge1$. In a conformal coordinate $z$ at $p_i$, its metric is comparable to $|z|^{-2(d_i+1)}|dz|^2$. Thus the interior ends can be included without additional hypotheses on the boundary ends.

\begin{theoremB}
Under the assumptions of Theorem~A, if $r\ge1$, then
\begin{equation}\label{eq:interior-main}
\Ind_s(\Sigma)\ge
\left\lceil\frac{2g+b-2+2\sum_{i=1}^r(d_i+1)}{3}\right\rceil.
\end{equation}
In particular, if all interior ends are embedded, then
$3\Ind_s(\Sigma)\ge 2g+b+4r-2$.
For $r=0$, Theorem~A applies.
\end{theoremB}

For a selection $S\subset\{1,\ldots,s\}$ of boundary ends, choose integers $n_j\ge1$, $j\in S$. A concrete sufficient hypothesis at a selected end is
\begin{equation}\label{eq:intro-growth}
\gamma=e^{2\lambda_j}|dz|^2,\qquad
 e^{\lambda_j(z)}\ge C_j^{-1}|z|^{-n_j},
\end{equation}
uniformly on a punctured closed half-disk. No such hypothesis is made at the unselected ends. Put
\[
L=2\sum_{i=1}^r(d_i+1)+\sum_{j\in S}n_j,
\qquad
\varepsilon_S=
\begin{cases}
0,&S\ne\{1,\ldots,s\},\\
1,&S=\{1,\ldots,s\}.
\end{cases}
\]

\begin{theoremC}
Under the assumptions of Theorem~A and \eqref{eq:intro-growth},
if $L\ge2$, then
\begin{equation}\label{eq:mixed-main}
\Ind_s(\Sigma)\ge
\left\lceil\frac{2g+b-2+L-\varepsilon_S}{3}\right\rceil.
\end{equation}
If, at some selected boundary end, the stronger lower bound
\(
e^{\lambda_j(z)}\ge C_j^{-1}|z|^{-\alpha_j}
\, \text{with }\alpha_j>n_j
\)
holds uniformly near the puncture, then $\varepsilon_S$ is zero in \eqref{eq:mixed-main}.
For $L\le1$, the estimate of Theorem~A applies.
\end{theoremC}

The second term in \eqref{eq:mixed-main} counts only the ends at which poles are allowed. An unselected boundary end is nevertheless useful: smooth extension there forces an exceptional form to vanish. The precise dimension formula and the construction of a single compactly supported cutoff sequence for mixed types of ends are proved below. An integral capacity condition can replace \eqref{eq:intro-growth}; see Lemma~\ref{lem:radial-capacity}. Embeddedness alone does not justify taking $n_j=2$, as the stable strip shows. Such an order is justified for embedded ends admitting reflection across one plane.

The coefficients in these inequalities come from the dimension of the chosen meromorphic-form space and the three Euclidean coordinate tests. We do not claim optimality or an equality classification. We also distinguish the estimates proved here from the remaining question of admitting poles at arbitrary boundary ends under mean-convexity alone.

\textbf{The paper is organized as follows}. In Sections~\ref{sec:preliminaries}--\ref{sec:hodge}, we fix our conventions, describe the conformal compactification, and compute the dimension of the space of absolute $L^2$ harmonic one-forms. In Sections~\ref{sec:energy}--\ref{sec:index}, we derive the localized energy identity, analyze the exceptional forms, and prove Theorem~A. In Section~\ref{sec:ends}, we establish Theorems~B and~C by counting meromorphic differentials and constructing cutoffs adapted to the different types of ends. Finally, in Section~\ref{sec:scope}, we present examples illustrating the hypotheses and an application to ends admitting reflection across a plane.

\section{Variational setting and conformal compactification}\label{sec:preliminaries}
Let $\gamma$ be the induced metric on $\Sigma$, let $N$ be a global unit normal, and set $A(V)=-D_VN$. We use $\Delta=\diver\nabla$. Along $\partial\Sigma$, let $\eta$ be the outward unit conormal and let $T$ be a unit tangent. The free boundary condition gives $dF(\eta)=\nu$. We identify tangent vectors with their images under $dF$ when taking Euclidean inner products. Set
\begin{equation}\label{eq:boundary-notation}
q=h^{\partial\Omega}(N,N),\qquad
k=h^{\partial\Omega}(T,T)=\langle\nabla_T\eta,T\rangle,
\qquad \mathscr H=k+q=H_{\partial\Omega}\circ F.
\end{equation}
In particular, $\mathscr H\ge0$ under our assumptions.

The Jacobi quadratic form is
\begin{equation}\label{eq:quadratic-form}
Q(f,v)=\int_\Sigma\bigl(\langle\nabla f,\nabla v\rangle-|A|^2fv\bigr)\dd A
-\int_{\partial\Sigma}qfv\dd s.
\end{equation}
Unless stated otherwise, at least one argument in this expression has compact support. All coefficients are smooth, so the pairing is then well defined for smooth functions, even if the other argument has no globally finite energy. Integration by parts gives
\[
Q(f,v)=-\int_\Sigma vJf\dd A+\int_{\partial\Sigma}vBf\dd s,
\qquad Jf=\Delta f+|A|^2f,\quad Bf=\partial_\eta f-qf.
\]
The notation $C_c^\infty(\Sigma)$ means smooth functions up to the boundary with compact support in the manifold with boundary. Such functions need not vanish on $\partial\Sigma$.

The strong index is the supremum of the dimensions of subspaces of $C_c^\infty(\Sigma)$ on which $Q$ is negative definite. The weak index is defined with the additional restriction $\int_\Sigma f\dd A=0$. Since that restriction has codimension at most one on any finite-dimensional space,
\begin{equation}\label{eq:weak-strong}
\Ind_w(\Sigma)\le\Ind_s(\Sigma)\le\Ind_w(\Sigma)+1.
\end{equation}
The same strong index is obtained by allowing compactly supported $H^1$ functions. Indeed, on a fixed compact region the trace theorem makes $Q$ continuous in $H^1$, and a basis of a finite-dimensional negative space can be approximated by smooth functions without losing negative definiteness.

\begin{proposition}[Hong--Saturnino]\label{prop:compactification}
Under the assumptions of Theorem~A, there exist a compact connected bordered Riemann surface $\Sigmabar$ and finite sets
$P_{\mathrm{int}}\subset\operatorname{int}\Sigmabar$ and
$P_{\partial}\subset\partial\Sigmabar$ such that $\Sigma$ is conformally equivalent to $\Sigmabar\setminus(P_{\mathrm{int}}\cup P_{\partial})$. Moreover, $P_{\partial}$ is nonempty and
\begin{equation}\label{eq:finite-curvature}
\int_\Sigma |A|^2\dd A+
\int_{\partial\Sigma}\mathscr H\dd s<\infty,
\qquad
\int_\Sigma |K|\dd A<\infty.
\end{equation}
\end{proposition}
\begin{proof}
Apply \cite[Theorem~1.4]{HS23} with contact angle $\pi/2$. Its interior curvature hypothesis holds because the ambient scalar curvature and the mean curvature of $\Sigma$ vanish. Its alternative boundary hypothesis reduces to $H_{\partial\Omega}\ge0$ along $\partial\Sigma$. Thus boundary compactness is not needed. In this description, interior punctures correspond to interior ends and boundary punctures correspond to boundary ends. If $P_{\partial}$ were empty, the boundary of $\Sigma$ would be diffeomorphic to the compact boundary of $\Sigmabar$, contrary to the hypothesis. The integral conclusion is also part of \cite[Theorem~1.4]{HS23}; the last assertion follows from $K=-|A|^2/2$.
\end{proof}

We henceforth identify $\Sigma$ with the punctured compactification. Choose any smooth metric $\gamma_0$ on $\Sigmabar$ compatible with its conformal structure. On $\Sigma$,
\[
\gamma=e^{2\rho}\gamma_0
\]
for a smooth function $\rho$, which need not extend across the punctures. All subsequent extensions concern differential forms and the conformal structure; the induced metric itself is not assumed to extend.

\section{Absolute harmonic forms and the topology of the compactification}\label{sec:hodge}
We work with real one-forms. Define
\[
\HH^1_{(2),\mathrm{abs}}(\Sigma)=
\left\{\omega\in C^\infty(T^*\Sigma):
 d\omega=0,\ \delta_\gamma\omega=0,\ \omega(\eta)=0,
 \ \int_\Sigma|\omega|_\gamma^2\dd A_\gamma<\infty\right\}.
\]
Smoothness here includes the geometric boundary, away from the removed points. For closed forms, the second absolute boundary condition $\iota_\eta d\omega=0$ is automatic. We use the analogous notation $\HH^1_{\mathrm{abs}}(\Sigmabar,\gamma_0)$ on the compactification.

\begin{lemma}\label{lem:conformal}
The conditions defining $\HH^1_{(2),\mathrm{abs}}$ are invariant under a conformal change of metric on an oriented surface. More precisely,
\begin{equation}\label{eq:conformal-norm}
|\omega|_\gamma^2\dd A_\gamma
=|\omega|_{\gamma_0}^2\dd A_{\gamma_0},
\qquad
\delta_\gamma\omega=e^{-2\rho}\delta_{\gamma_0}\omega,
\qquad
\omega(\eta_\gamma)=e^{-\rho}\omega(\eta_{\gamma_0}).
\end{equation}
\end{lemma}
\begin{proof}
Since $\gamma=e^{2\rho}\gamma_0$, the inverse metric and area form satisfy
$\gamma^{-1}=e^{-2\rho}\gamma_0^{-1}$ and
$dA_\gamma=e^{2\rho}dA_{\gamma_0}$, proving the first identity.
On $k$-forms in dimension two, the defining relation
$\alpha\wedge\star_\gamma\beta
=\langle\alpha,\beta\rangle_\gamma\,dA_\gamma$ gives
$\star_\gamma=e^{(2-2k)\rho}\star_{\gamma_0}$.
Thus $\delta_\gamma=-\star_\gamma d\star_\gamma$ on one-forms
yields $\delta_\gamma\omega=e^{-2\rho}\delta_{\gamma_0}\omega$.
Finally, conformal changes preserve orthogonality and outward direction,
so normalization gives $\eta_\gamma=e^{-\rho}\eta_{\gamma_0}$.
Closedness is metric-independent because the exterior derivative is.
\end{proof}

On interior removed points we have the following result.
\begin{lemma}\label{lem:holomorphic-removal}
If $f$ is holomorphic on a punctured disk $D^*=\{0<|z|<\rho\}$ and
$\int_{D^*}|f|^2\dd x\dd y<\infty$, then $f$ extends holomorphically across the origin.
\end{lemma}
\begin{proof}
Write the Laurent expansion $f(z)=\sum_{m\in\mathbb Z}a_mz^m$. On every circle contained in $D^*$, Parseval's identity yields
\[
\int_0^{2\pi}|f(re^{it})|^2\dd t
=2\pi\sum_{m\in\mathbb Z}|a_m|^2r^{2m}.
\]
For $m\le-1$, a nonzero coefficient $a_m$ would imply
\[
\int_{D^*}|f|^2\dd x\dd y
\ge 2\pi|a_m|^2\int_0^\rho r^{2m+1}\dd r=\infty.
\]
All negative Laurent coefficients therefore vanish.
\end{proof}

Next, we present an extension of the forms along the interior and boundary punctures.
\begin{proposition}\label{prop:hodge-extension}
Let $\Sigmabar$ be a compact connected bordered Riemann surface, and remove finitely many interior and boundary points to obtain $\Sigma$. Equip $\Sigma$ with any smooth metric conformal to a smooth compatible metric $\gamma_0$ on $\Sigmabar$. Restriction induces an isomorphism
\begin{equation}\label{eq:hodge-isomorphism}
\HH^1_{\mathrm{abs}}(\Sigmabar,\gamma_0)
\longrightarrow\HH^1_{(2),\mathrm{abs}}(\Sigma).
\end{equation}
\end{proposition}
\begin{proof}
By Lemma~\ref{lem:conformal}, it suffices to use $\gamma_0$. We first prove that a form on the right extends through every puncture.

\emph{Interior punctures.}
Take a conformal coordinate $z=x+iy$ centered at an interior puncture and write
\[
\gamma_0=e^{2\lambda}(\dd x^2+\dd y^2),\qquad
\omega=a\dd x+b\dd y.
\]
The equations $d\omega=0$ and $\delta\omega=0$ read
\[
b_x-a_y=0,\qquad a_x+b_y=0.
\]
So, $f=a-ib$ is holomorphic and
$\omega=\Rea(f\dd z)$. Moreover,
\begin{equation}\label{eq:local-l2}
|\omega|_{\gamma_0}^2\dd A_{\gamma_0}
=(a^2+b^2)\dd x\dd y=|f|^2\dd x\dd y.
\end{equation}
Lemma~\ref{lem:holomorphic-removal} extends $f$ holomorphically, hence $\omega$ smoothly and harmonically, across the puncture.

\emph{Boundary punctures.}
Choose a boundary conformal coordinate in which a neighborhood of the puncture is the upper half-disk, the puncture is $0$, and the geometric boundary corresponds to the real axis with $0$ removed. As above, write $\omega=\Rea(f\dd z)$ with $f=a-ib$ holomorphic. Along the real axis the outward conormal is $-e^{-\lambda}\partial_y$. Thus the absolute condition is exactly
\[
b(x,0)=0,
\]
so $f$ is real on both real boundary intervals. Schwarz reflection defines
\[
\widetilde f(z)=
\begin{cases}
f(z),&\operatorname{Im}z\ge0,\\
\overline{f(\overline z)},&\operatorname{Im}z<0.
\end{cases}
\]
It is holomorphic on the punctured full disk: reflection applies near every point of the real axis other than the puncture. By \eqref{eq:local-l2}, its squared integral is twice that over the upper half-disk, and is finite. Lemma~\ref{lem:holomorphic-removal} removes the remaining singularity. The reflected symmetry persists at $0$, so the extension is real on the entire real interval and the extended form still satisfies the absolute boundary condition.

The extensions are unique and agree on overlaps, since they agree on the original punctured surface. They therefore give a smooth absolute harmonic form on $\Sigmabar$.

Conversely, a smooth absolute harmonic form on $\Sigmabar$ has finite $L^2(\gamma_0)$ norm by compactness. Its restriction has finite $L^2(\gamma)$ norm by \eqref{eq:conformal-norm}, and remains closed, coclosed, and absolute. Finally, restriction is injective because the punctured surface is dense. This proves \eqref{eq:hodge-isomorphism}.
\end{proof}

\begin{corollary}\label{cor:hodge-dimension}
If $\Sigmabar$ has genus $g$ and $b\ge1$ boundary components, then
\begin{equation}\label{eq:hodge-dimension}
\dim_{\R}\HH^1_{(2),\mathrm{abs}}(\Sigma)=2g+b-1.
\end{equation}
\end{corollary}
\begin{proof}
The Hodge theorem with absolute boundary conditions gives
\[
\HH^1_{\mathrm{abs}}(\Sigmabar,\gamma_0)
\cong H^1_{\mathrm{dR}}(\Sigmabar;\R);
\]
see \cite{Schwarz95}. Since $\Sigmabar$ is connected and has nonempty boundary, its zeroth Betti number is one and its second Betti number is zero. Hence
\[
2-2g-b=\chi(\Sigmabar)=1-\dim H^1_{\mathrm{dR}}(\Sigmabar;\R).
\]
Combine this identity with Proposition~\ref{prop:hodge-extension}.
\end{proof}

\begin{remark}
The isomorphism for interior punctures is the compact-boundary situation appearing in \cite{CMS26}. The additional issue here is extension at boundary punctures. The formula counts the topology of $\Sigmabar$, and it is independent of both puncture sets. When $b=0$, the corresponding closed-surface formula is $2g$; that case is outside Theorem~A. The Hodge star exchanges absolute and relative harmonic one-forms on a surface, so their $L^2$ dimensions agree.
\end{remark}
We will also use boundary uniqueness for harmonic one-forms.
\begin{lemma}\label{lem:boundary-uniqueness}
Let $\omega$ be a smooth closed and coclosed real one-form on a connected surface with boundary. If $\omega$ vanishes on a nonempty open boundary arc as a covector, then $\omega$ vanishes everywhere.
\end{lemma}
\begin{proof}
In a boundary conformal chart, $\omega=\Rea(f\dd z)$ with $f$ holomorphic. Vanishing of the covector implies $f=0$ on the arc. Reflect $f$ across the real axis, where its boundary values are real. The reflected holomorphic function vanishes on an interval in its interior, and hence vanishes identically near that interval. The identity theorem, applied in overlapping conformal charts, propagates the vanishing over the connected surface.
\end{proof}

\section{A localized identity for the Jacobi quadratic form}\label{sec:energy}
For any smooth closed and coclosed real one-form $\omega$ satisfying the absolute boundary condition, set
\[
X=\omega^\sharp,\qquad u_i=\langle X,e_i\rangle,\quad i=1,2,3,
\]
where $e_1,e_2,e_3$ is the standard Euclidean basis. In particular,
$\sum_i u_i^2=|X|^2=|\omega|^2$.

\begin{lemma}\label{lem:boundary-bochner}
Along $\partial\Sigma$,
\[
\frac12\partial_\eta|X|^2=-k|X|^2.
\]
\end{lemma}
\begin{proof}
Since $\omega(\eta)=0$, write $X=fT$ on the boundary. Closedness of $\omega$ implies symmetry of $\nabla X$. Therefore,
\[
\langle\nabla_\eta X,T\rangle
=\langle\nabla_T X,\eta\rangle
=-\langle X,\nabla_T\eta\rangle=-kf.
\]
Taking the inner product of $\nabla_\eta X$ with $X=fT$ proves the identity, including at zeros of $X$.
\end{proof}

The following identity requires no global integrability assumption on $\omega$.

\begin{proposition}\label{prop:energy}
For every compactly supported Lipschitz function $\varphi$ on $\Sigma$,
\begin{equation}\label{eq:localized}
\sum_{i=1}^3Q(\varphi u_i,\varphi u_i)
=\int_\Sigma|\nabla\varphi|^2|\omega|^2\dd A
-\int_{\partial\Sigma}\mathscr H\varphi^2|\omega|^2\dd s.
\end{equation}
\end{proposition}
\begin{proof}
We first take $\varphi$ smooth. For a local orthonormal tangent frame $E_1,E_2$, the Gauss formula gives
\[
\sum_{i=1}^3|\nabla u_i|^2
=\sum_{\alpha=1}^2|D_{E_\alpha}X|^2
=|\nabla X|^2+|AX|^2.
\]
Minimality in dimension two implies
$A^2=(|A|^2/2)\operatorname{Id}$ and $K=-|A|^2/2$. Hence
\[
\sum_i\bigl(|\nabla u_i|^2-|A|^2u_i^2\bigr)
=|\nabla X|^2-\frac{|A|^2}{2}|X|^2
=\frac12\Delta|X|^2,
\]
where the last equality is the Bochner formula for a harmonic one-form. Expanding the gradients of $\varphi u_i$ now yields
\begin{align*}
\sum_iQ(\varphi u_i,\varphi u_i)
={}&\int_\Sigma|\nabla\varphi|^2|X|^2\dd A
+\int_\Sigma\varphi\langle\nabla\varphi,\nabla|X|^2\rangle\dd A\\
&+\frac12\int_\Sigma\varphi^2\Delta|X|^2\dd A
-\int_{\partial\Sigma}q\varphi^2|X|^2\dd s.
\end{align*}
Integration by parts in the third term cancels the second term and leaves
\[
\int_\Sigma|\nabla\varphi|^2|X|^2\dd A
+\int_{\partial\Sigma}\varphi^2
\left(\frac12\partial_\eta|X|^2-q|X|^2\right)\dd s.
\]
Lemma~\ref{lem:boundary-bochner} and \eqref{eq:boundary-notation} give \eqref{eq:localized}. Local $H^1$ approximation and the trace theorem extend the identity to compactly supported Lipschitz cutoffs.
\end{proof}

Fix $p\in\Sigma$. Completeness and local compactness make the intrinsic length metric proper, so its closed bounded balls are compact. For $R>0$, use
\[
\varphi_R(x)=\chi\bigl(d_\Sigma(p,x)/R\bigr),
\]
where $\chi$ is Lipschitz, equals one on $[0,1]$, vanishes on $[2,\infty)$, and takes values in $[0,1]$. Then
\begin{equation}\label{eq:cutoffs}
\supp\varphi_R\subset\overline B_{2R},\qquad
\varphi_R=1\text{ on }B_R,\qquad
|\nabla\varphi_R|\le C/R\text{ a.e.}
\end{equation}
For $\omega\in\HH^1_{(2),\mathrm{abs}}(\Sigma)$, the harmonic-form error satisfies
\begin{equation}\label{eq:cutoff-error}
\int_\Sigma|\nabla\varphi_R|^2|\omega|^2\dd A
\le\frac{C^2}{R^2}\|\omega\|_{L^2}^2\longrightarrow0.
\end{equation}
No condition on $\partial_\eta\varphi_R$ is imposed. The tests are admissible for the quadratic form without satisfying Robin conditions individually; those conditions arise later from a weak equation. Nor do we assume that $\mathscr H|\omega|^2$ is integrable on the entire boundary. Every boundary integral in \eqref{eq:localized} has compact support, and its sign will suffice.

\section{Exceptional forms and a boundary-end obstruction}\label{sec:exceptional}
Define
\begin{equation}\label{eq:exceptional-space}
\KK=\left\{\omega\in\HH^1_{(2),\mathrm{abs}}(\Sigma):
Ju_i=0\text{ in }\Sigma,\quad Bu_i=0\text{ on }\partial\Sigma,
\quad i=1,2,3\right\}.
\end{equation}
Only this space, rather than the full Jacobi kernel, is relevant to the equality case of the harmonic-form construction.

\begin{lemma}\label{lem:constant}
If $\omega\in\KK$, then for every connected boundary component $\Gamma$ there is a real constant $c_\Gamma$ such that
\begin{equation}\label{eq:constant-component}
X=c_\Gamma T\quad\text{on }\Gamma.
\end{equation}
In particular, if $\partial\Sigma\ne\varnothing$, then $\dim\KK\le1$.
\end{lemma}
\begin{proof}
Since the Euclidean frame is parallel, the three scalar Robin equations
combine into
\(
D_\eta X=qX.
\)
The absolute boundary condition gives $X=fT$ on $\Gamma$. Taking the
inner product with $\eta$ and using the Gauss formula, we obtain
$\langle\nabla_\eta X,\eta\rangle=0$. Moreover, $\delta\omega=0$
implies $\diver X=0$, and hence
\[
0=\langle\nabla_TX,T\rangle+\langle\nabla_\eta X,\eta\rangle
=T(f),
\]
where we used $|T|=1$. Since $\Gamma$ is connected, $f=c_\Gamma$
is constant. For a fixed $\Gamma$, the linear map
$\omega\mapsto c_\Gamma$ is injective: if $c_\Gamma=0$, then
$X=0$ on $\Gamma$, so $\omega$ vanishes there as a covector.
Lemma~\ref{lem:boundary-uniqueness} yields $\omega\equiv0$.
Thus $\dim\mathcal K\le1$.
\end{proof}

\begin{remark}
The proof of Lemma~\ref{lem:constant} also applies to forms that are not globally $L^2$. It is the intrinsic version of the boundary constancy argument used in \cite{Mendes26}. Its proof does not require $\Sigma$ to be flat. It only uses the tangency and divergence equations and the common scalar Robin coefficient of the three Euclidean components.
\end{remark}

\begin{lemma}\label{lem:finite-integral}
Let $\omega\in\HH^1_{(2),\mathrm{abs}}(\Sigma)$, and let $\alpha:(0,\varepsilon]\to\partial\Sigma$ be a boundary arc that, in a boundary coordinate of $\Sigmabar$, approaches a point of $P_{\partial}$ as $t\downarrow0$. Then
\[
\int_0^\varepsilon|\omega(\alpha'(t))|\dd t<\infty.
\]
If the induced metric on $\Sigma$ is complete, the same arc has infinite induced length towards that puncture.
\end{lemma}
\begin{proof}
By Proposition~\ref{prop:hodge-extension}, $\omega$ extends smoothly across the puncture. Parametrize the arc by a smooth boundary coordinate $t$ of the compactification. Its pullback is $a(t)\dd t$, where $a$ extends smoothly to $t=0$, and hence is integrable.

If the induced length towards the puncture were finite, then the length of each sufficiently small tail would tend to zero. Therefore $\alpha(t)$ would be Cauchy as $t\downarrow0$. Completeness would give a limit in $\Sigma$. The metric topology agrees with the manifold topology, and the inclusion into $\Sigmabar$ is continuous, so that limit would also be the puncture. This is impossible because the puncture does not belong to $\Sigma$.
\end{proof}

Next, we establish a key result for surfaces with noncompact boundary.

\begin{proposition}\label{prop:vanishing}
Suppose that $\Sigma$ is complete, has the finite conformal compactification of Proposition~\ref{prop:compactification}, and has noncompact boundary. Then $\KK=\{0\}$.
\end{proposition}
\begin{proof}
Choose a boundary puncture and a neighboring boundary arc $\alpha$ as in Lemma~\ref{lem:finite-integral}. For $\omega\in\KK$, let $\Gamma$ be the boundary component containing this arc. By Lemma~\ref{lem:constant}, $X=c_\Gamma T$ along $\Gamma$. Therefore, independently of the parametrization,
\[
|\omega(\alpha'(t))|=|c_\Gamma|\,|\alpha'(t)|_\gamma.
\]
The integral of the left-hand side is finite, while the integral of $|\alpha'|_\gamma$ is infinite by Lemma~\ref{lem:finite-integral}. Thus $c_\Gamma=0$, and boundary uniqueness gives $\omega=0$.
\end{proof}

This argument uses a finite absolute integral of a one-form along a compactified boundary arc. It does not infer an $L^2$ trace estimate in the induced metric from interior $L^2$ integrability. The smooth extension at the boundary puncture is what justifies the finite integral.

\section{The index estimate and the role of nullity}\label{sec:index}
We now obtain the Jacobi equations from the vanishing of localized energies. The use of a fixed compactly supported negative space avoids choosing global negative eigenfunctions or imposing global energy assumptions on the uncut coordinate functions.

\begin{lemma}\label{lem:negative-space}
Assume $I=\Ind_s(\Sigma)<\infty$. There is a subspace
$\EE=\spanop\{\psi_1,\ldots,\psi_I\}\subset C_c^\infty(\Sigma)$
of dimension $I$ on which $Q$ is negative definite. On
\[
\EE^{\perp_Q}=
\{v\in H_c^1(\Sigma):Q(v,\psi_j)=0\text{ for all }j\},
\]
$Q$ is nonnegative and satisfies
\[
|Q(v,w)|^2\le Q(v,v)Q(w,w).
\]
Every compactly supported smooth $v$ has a decomposition
$v=v^\perp+e$ with $e\in\EE$ and $v^\perp\in\EE^{\perp_Q}\cap C_c^\infty(\Sigma)$.
\end{lemma}
\begin{proof}
Since the index is the finite supremum of integer dimensions, there
exists a negative definite subspace $\EE$ of dimension $I$; if $I=0$,
take $\EE=\{0\}$. If $v\in\EE^{\perp_Q}$ satisfied $Q(v,v)<0$, then
\[
Q(e+tv,e+tv)=Q(e,e)+t^2Q(v,v),
\qquad \text{for}\, e\in\EE,\ t\in\R
\]
would make $\EE\oplus\R v$ negative definite, contradicting the
$H_c^1$ characterization of the index. Thus $Q$ is nonnegative on
$\EE^{\perp_Q}$, and applying this to $v+tw$ gives Cauchy--Schwarz.

Finally, the matrix $(Q(\psi_i,\psi_j))$ is negative definite and
therefore invertible. For any $v\in C_c^\infty(\Sigma)$, the system
\[
\sum_{i=1}^I a_iQ(\psi_i,\psi_j)=Q(v,\psi_j),
\qquad j=1,\ldots,I,
\]
has a unique solution. Setting $e=\sum_i a_i\psi_i$ gives the desired
decomposition $v=(v-e)+e$, with $v-e\in\EE^{\perp_Q}$.
\end{proof}

The following result bounds the dimension of the space of harmonic one-forms in terms of the Morse index and the dimension of the exceptional space.

\begin{proposition}\label{prop:index-reduction}
Let $\Sigma$ be a complete free boundary minimal surface with finite strong index $I$ and $\mathscr H\ge0$. For the space $\KK$ in \eqref{eq:exceptional-space},
\begin{equation}\label{eq:dimension-reduction}
\dim\HH^1_{(2),\mathrm{abs}}(\Sigma)\le3I+\dim\KK,
\end{equation}
whenever the right-hand side is finite.
\end{proposition}
\begin{proof}
Choose $\EE$ as in Lemma~\ref{lem:negative-space}. Define
\[
\mathcal F:\HH^1_{(2),\mathrm{abs}}(\Sigma)\longrightarrow\R^{3I},
\qquad
\mathcal F(\omega)=\bigl(Q(u_i,\psi_j)\bigr)_{i,j}.
\]
These pairings are well defined because the $\psi_j$ have compact support. We claim that $\ker\mathcal F\subset\KK$.

Fix $\omega\in\ker\mathcal F$. For all sufficiently large $R$, the cutoff $\varphi_R$ from \eqref{eq:cutoffs} is one on a neighborhood of $\bigcup_j\supp\psi_j$. Thus
\[
Q(\varphi_Ru_i,\psi_j)=Q(u_i,\psi_j)=0,
\qquad Q(\varphi_Ru_i,\varphi_Ru_i)\ge0.
\]
Combining the sign of $\mathscr H$ with Proposition~\ref{prop:energy} gives, for each $i$,
\begin{equation}\label{eq:vanishing-energy}
0\le Q(\varphi_Ru_i,\varphi_Ru_i)
\le\sum_{\ell=1}^3Q(\varphi_Ru_\ell,\varphi_Ru_\ell)
\le\frac{C^2}{R^2}\|\omega\|_{L^2}^2\longrightarrow0.
\end{equation}
There is no passage to the limit in an unsigned global boundary integral here.

For arbitrary $v\in C_c^\infty(\Sigma)$, decompose $v=v^\perp+e$ as in Lemma~\ref{lem:negative-space}. Cauchy--Schwarz on the nonnegative complement yields
\[
|Q(\varphi_Ru_i,v^\perp)|^2
\le Q(\varphi_Ru_i,\varphi_Ru_i)Q(v^\perp,v^\perp)
\longrightarrow0.
\]
Also $Q(\varphi_Ru_i,e)=0$. Once $R$ is large enough that $\varphi_R=1$ near $\supp v$, locality of the pairing gives
\[
Q(u_i,v)=Q(\varphi_Ru_i,v)\longrightarrow0.
\]
Therefore $Q(u_i,v)=0$ for all compactly supported smooth $v$. Tests supported in the interior imply $Ju_i=0$. Integration by parts with arbitrary compactly supported boundary traces then gives $Bu_i=0$. The functions $u_i$ are smooth up to the geometric boundary and belong to $L^2$ because $|u_i|\le|\omega|$. This proves the claim.

Rank--nullity now gives \eqref{eq:dimension-reduction}. More explicitly, apply the preceding construction to any finite-dimensional subspace of the harmonic-form space. Its dimension is at most $3I+\dim\KK$, which also proves finiteness when needed.
\end{proof}

\begin{proof}[Proof of Theorem~A]
Proposition~\ref{prop:compactification} supplies the compactification and at least one boundary puncture. Corollary~\ref{cor:hodge-dimension} gives
$\dim\HH^1_{(2),\mathrm{abs}}(\Sigma)=2g+b-1$.
Proposition~\ref{prop:vanishing} gives $\KK=\{0\}$. Applying Proposition~\ref{prop:index-reduction}, we obtain $2g+b-1\le3\Ind_s(\Sigma)$. Since the index is an integer, this is \eqref{eq:main}. The weak-index inequality follows from \eqref{eq:weak-strong}.
\end{proof}

\subsection*{The role of the exceptional harmonic forms}
For clarity, define the $L^2$ Jacobi kernel through compactly supported tests:
\[
\NN_{(2)}(J)=\{f\in L^2(\Sigma)\cap C^\infty(\Sigma):
Q(f,v)=0\text{ for every }v\in C_c^\infty(\Sigma)\},
\qquad \nu_{(2)}(J)=\dim\NN_{(2)}(J).
\]
Equivalently, $Jf=0$ and $Bf=0$. The definition does not require the separate global integrals defining $Q(f,f)$ to converge. The map
\[
\KK\longrightarrow\NN_{(2)}(J)^3,\qquad
\omega\longmapsto(u_1,u_2,u_3)
\]
is injective. Using only this fact in \eqref{eq:dimension-reduction} would give
\[
\Ind_s(\Sigma)+\nu_{(2)}(J)
\ge\tfrac13\dim\HH^1_{(2),\mathrm{abs}}(\Sigma)
\]
when the nullity is finite. The boundary argument proves the stronger statement $\KK=\{0\}$ under the assumptions of Theorem~A, regardless of the size of the full Jacobi kernel. It does not assert $\nu_{(2)}(J)=0$.

For comparison, Lemma~\ref{lem:constant} alone bounds $\dim\KK$ by one whenever the geometric boundary is nonempty. With compact boundary, this observation would yield only
$3\Ind_s(\Sigma)\ge\dim\HH^1_{(2),\mathrm{abs}}(\Sigma)-1$
by the present argument. This is an intermediate consequence of our method, not a replacement for the compact-boundary results of \cite{CMS26}. The boundary-end obstruction is the step that removes this remaining loss in the setting of Theorem~A.

\section{Interior ends and selected boundary ends}\label{sec:ends}
Set $G=2g+b-1$. Throughout this section, the metric is the original induced metric $\gamma$, and $s=|P_{\partial}|\ge1$. All meromorphic extensions below are on the conformal compactification or its double, not extensions of the minimal immersion across the supporting surface.

\subsection{The local geometry of interior ends}

The next result is local and follows essentially from the argument in \cite[Section~2.4]{CM23}.

\begin{lemma}\label{lem:interior-asymptotics}
For each interior puncture $p_i$, there is an integer $d_i\ge1$ and a conformal coordinate $z$ centered there such that
\[
\gamma=e^{2\lambda}|dz|^2,\qquad
 e^\lambda\asymp |z|^{-d_i-1}.
\]
If the end is embedded, then $d_i=1$. Every meromorphic one-form with pole order at most $d_i+1$ has bounded geometric norm near that puncture.
\end{lemma}
\begin{proof}
The end is a punctured disk, is complete towards its puncture, and has finite total curvature by \eqref{eq:finite-curvature}. The classical local finite-total-curvature end theorem gives meromorphic Weierstrass differentials with a pole of maximal order $d_i+1\ge2$; the end multiplicity is $d_i$, and an embedded end has multiplicity one. This is the local end description used in \cite[Section~2.4]{CM23}; it concerns a single end and does not require the other ends of $\Sigma$ to avoid the boundary. To see the metric estimate explicitly, write the vector of Weierstrass differentials as
\[
\Phi(z)=\bigl(a z^{-d_i-1}+O(z^{-d_i})\bigr)\,dz,
\qquad a\in\C^3\setminus\{0\}.
\]
Since $dF=\Rea\Phi$, conformality gives $\gamma=\tfrac12|\Phi/dz|^2|dz|^2$. The nonzero leading vector yields upper and lower bounds by positive multiples of $|z|^{-2d_i-2}$. If $\omega=\Rea(f\,dz)$ and $f=O(|z|^{-d_i-1})$, then $|\omega|_\gamma=e^{-\lambda}|f|=O(1)$.
\end{proof}

\subsection{Riemann--Roch with selected boundary poles}

Choose $S\subset\{1,\ldots,s\}$ and positive integers $n_j$,
$j\in S$. Let $\mathcal V$ be the real space of absolute harmonic
one-forms on $\Sigma$ whose local holomorphic representatives have
poles of order at most $d_i+1$ at $p_i$, poles of order at most
$n_j$ at $q_j$ for $j\in S$, and removable singularities at $q_j$
for $j\notin S$. Pole order zero means a removable singularity. Put
\[
L=2\sum_{i=1}^r(d_i+1)+\sum_{j\in S}n_j.
\]

\begin{lemma}\label{lem:meromorphic-dimension}
The space $\mathcal V$ is finite dimensional, with
\begin{equation}\label{eq:meromorphic-dimension}
\dim_{\R}\mathcal V=
\begin{cases}
G,&L=0,\\
G-1+L,&L>0,
\end{cases}
\qquad G=2g+b-1.
\end{equation}
\end{lemma}

\begin{proof}
Let $X=\mathcal D\Sigmabar$ be the conformal double, with
antiholomorphic involution $\sigma$. Since the components along which
we double are circles,
\[
\chi(X)=2\chi(\Sigmabar)=4-4g-2b.
\]
Thus $X$ has genus $G=2g+b-1$. Each interior puncture $p_i$
gives two points $p_i^+,p_i^-$ exchanged by $\sigma$, while each
boundary puncture $q_j$ gives one point fixed by $\sigma$. Define
the effective $\sigma$-invariant divisor
\[
D=\sum_{i=1}^r(d_i+1)(p_i^++p_i^-)
  +\sum_{j\in S}n_jq_j.
\]
In particular,
\(
\deg D
=2\sum_{i=1}^r(d_i+1)+\sum_{j\in S}n_j=L.
\)
We first identify $\mathcal V$ with the fixed space of the
antilinear involution
\[
\mathcal T\alpha=\overline{\sigma^*\alpha}
\quad\text{on}\quad H^0(X,K_X+D).
\]
Away from the punctures, write an absolute harmonic form locally
as $\omega=\Rea(f(z)\,dz)$, where $f$ is holomorphic. In a boundary
chart with $\sigma(z)=\bar z$ and boundary given by $\Ima z=0$, the
absolute condition says that the normal component of $\omega$
vanishes. Since
\[
\Rea(f\,dz)=(\Rea f)\,dx-(\Ima f)\,dy,
\]
this is equivalent to $\Ima f(x)=0$ along the regular boundary.
Schwarz reflection therefore extends $f(z)\,dz$ to the other copy
by the rule $f(\bar z)=\overline{f(z)}$. The reflected local
differentials agree on overlaps, giving a differential $\alpha$
on the double with $\mathcal T\alpha=\alpha$.

The same argument applies in a punctured boundary chart centered
at $q_j$ because reflection across the two adjacent regular boundary
arcs gives a holomorphic differential on a punctured disk. Its
singularity at the center has order at most $n_j$ if $j\in S$
and is removable if $j\notin S$, precisely by the pole condition
in the definition of $\mathcal V$. At an interior puncture, the
representative on the original copy has order at most $d_i+1$;
reflection gives the same bound at its distinct lift on the other
copy. Hence $\alpha\in H^0(X,K_X+D)$.

Conversely, if $\alpha\in H^0(X,K_X+D)$ is fixed by $\mathcal T$,
then $\omega=\Rea\alpha$ on the original copy is harmonic. In a
boundary chart, invariance gives
$f(\bar z)=\overline{f(z)}$, so $f$ is real on the regular boundary
and $\omega$ has vanishing normal component there. The divisor
bound gives exactly the pole orders and removable singularities
specified in the definition of $\mathcal V$. These two
constructions are inverse. In fact, if $\Rea(f\,dz)=0$ on an open set, then
both the $dx$ and $dy$ coefficients vanish, and hence $f=0$.
Therefore,
\[
\mathcal V\cong_{\R}
H^0(X,K_X+D)^{\mathcal T}.
\]

For any complex vector space with an antilinear involution, its
fixed space has real dimension equal to its complex dimension.
Indeed, every $\alpha$ decomposes uniquely as
\[
\alpha=
\frac{\alpha+\mathcal T\alpha}{2}
+i\,\frac{\alpha-\mathcal T\alpha}{2i},
\]
where both terms before multiplication by $i$ are fixed by
$\mathcal T$. It remains to compute $h^0(X,K_X+D)$.

If $L>0$, Riemann--Roch yields
\[
h^0(X,K_X+D)=G-1+L+h^0(X,\mathcal O(-D))
            =G-1+L.
\]
Indeed, a section of $\mathcal O(-D)$ is a holomorphic function
on the compact connected surface $X$ that vanishes at a point
of the nonzero effective divisor $D$, and thus is identically
zero. If $L=0$, then $D=0$ and $h^0(X,K_X)=G$. The claimed
formula follows.
\end{proof}

The subtraction by one when $L>0$ is essential. In particular,
the first simple boundary pole need not increase the dimension, since
on a doubled sphere a meromorphic differential cannot have just
one simple pole with nonzero residue. Riemann--Roch incorporates
this global compatibility automatically.

\subsection{Local cutoffs and their global combination}
\begin{lemma}\label{lem:radial-capacity}
On a punctured disk or half-disk with metric $e^{2\lambda}|dz|^2$, let $m\ge1$ and set
\[
a_m(r)=r^{1-2m}\int_{\Theta}e^{-2\lambda(r,\theta)}\,d\theta,
\]
where $\Theta=[0,2\pi]$ or $[0,\pi]$, respectively. If
\begin{equation}\label{eq:capacity-condition}
\int_0^{r_0}\frac{dr}{a_m(r)}=\infty,
\end{equation}
there are radial cutoffs $\zeta_\ell$, zero near the puncture and equal to one outside neighborhoods shrinking to it, such that
\[
\int |\nabla\zeta_\ell|_\gamma^2|\omega|_\gamma^2\,dA_\gamma\longrightarrow0
\]
for every form whose holomorphic representative has pole order at most $m$. In particular, the lower bound $e^\lambda\ge C^{-1}r^{-m}$ implies this conclusion.
\end{lemma}
\begin{proof}
Write $\omega=\Rea(f\,dz)$. Since $f$ has a pole of order at
most $m$, after reducing the coordinate neighborhood if necessary,
\(
|f(z)|\le C_\omega |z|^{-m}.
\)
For a radial function $\zeta=\zeta(r)$, the metric
$\gamma=e^{2\lambda}(dr^2+r^2d\theta^2)$ gives
\[
|\nabla\zeta|_\gamma^2=e^{-2\lambda}|\zeta'(r)|^2,
\qquad
|\omega|_\gamma^2=e^{-2\lambda}|f|^2,
\qquad
dA_\gamma=e^{2\lambda}r\,dr\,d\theta.
\]
So, if $\zeta'$ is supported in
$[\varepsilon,\delta]$, then
\begin{align*}
\int |\nabla\zeta|_\gamma^2|\omega|_\gamma^2\,dA_\gamma
&=\int_\varepsilon^\delta |\zeta'(r)|^2r
  \int_\Theta e^{-2\lambda(r,\theta)}
  |f(re^{i\theta})|^2\,d\theta\,dr\\
&\le C_\omega^2
  \int_\varepsilon^\delta|\zeta'(r)|^2a_m(r)\,dr.
\end{align*}

We next minimize this one-dimensional upper bound. The function
$a_m$ is positive and continuous on every compact subinterval
of $(0,r_0)$. For any absolutely continuous $\zeta$ satisfying
$\zeta(\varepsilon)=0$ and $\zeta(\delta)=1$,
Cauchy--Schwarz yields
\[
1
=\left|\int_\varepsilon^\delta\zeta'(r)\,dr\right|^2
\le
\left(\int_\varepsilon^\delta
|\zeta'(r)|^2a_m(r)\,dr\right)
\left(\int_\varepsilon^\delta a_m(r)^{-1}\,dr\right).
\]
Thus the radial energy is bounded below by
\[
A_{\varepsilon,\delta}^{-1},
\qquad
A_{\varepsilon,\delta}
:=\int_\varepsilon^\delta a_m(r)^{-1}\,dr.
\]
Equality holds for the increasing function
\[
\zeta_{\varepsilon,\delta}(r)
=
\frac{1}{A_{\varepsilon,\delta}}
\int_\varepsilon^r a_m(t)^{-1}\,dt,
\qquad \varepsilon\le r\le\delta,
\]
since its derivative is
$A_{\varepsilon,\delta}^{-1}a_m(r)^{-1}$.

Choose $\delta_\ell\downarrow0$. By
\eqref{eq:capacity-condition}, and because $a_m^{-1}$ is
integrable on compact subintervals of $(0,r_0)$,
\[
\int_0^{\delta_\ell}a_m(r)^{-1}\,dr=\infty.
\]
We may therefore choose $0<\varepsilon_\ell<\delta_\ell$ such that
$A_{\varepsilon_\ell,\delta_\ell}\ge\ell$. Extend
$\zeta_{\varepsilon_\ell,\delta_\ell}$ by zero for
$r\le\varepsilon_\ell$ and by one for $r\ge\delta_\ell$,
and denote the resulting function by $\zeta_\ell$.
It takes values in $[0,1]$ and is Lipschitz: its derivative is
bounded on the compact transition interval and vanishes elsewhere.
Moreover, $\zeta_\ell$ is eventually one on every compact subset
away from the puncture, and
\[
\int |\nabla\zeta_\ell|_\gamma^2|\omega|_\gamma^2\,dA_\gamma
\le \frac{C_\omega^2}{A_{\varepsilon_\ell,\delta_\ell}}
\le \frac{C_\omega^2}{\ell}\longrightarrow0.
\]

Finally, if $e^\lambda\ge C^{-1}r^{-m}$ uniformly in $\theta$,
then
\[
a_m(r)
=r^{1-2m}\int_\Theta e^{-2\lambda(r,\theta)}\,d\theta
\le C^2|\Theta|\,r.
\]
Hence $\int_0^{r_0}a_m(r)^{-1}\,dr=\infty$.
In this case one may also use the explicit logarithmic transition
\[
\zeta_\varepsilon(r)=
\frac{\log(r/\varepsilon^2)}{|\log\varepsilon|},
\qquad
\varepsilon^2\le r\le\varepsilon,
\]
extended by zero and one. Indeed,
\[
\int |\nabla\zeta_\varepsilon|_\gamma^2
|\omega|_\gamma^2\,dA_\gamma
\le
\frac{C_\omega^2C^2|\Theta|}{|\log\varepsilon|^2}
\int_{\varepsilon^2}^{\varepsilon}\frac{dr}{r}
=
\frac{C_\omega^2C^2|\Theta|}{|\log\varepsilon|}
\longrightarrow0.
\]
\end{proof}

\begin{remark}\label{rem:power-log}
More generally, a uniform lower bound $e^\lambda\ge\Lambda(r)$ suffices when
\[
\int_0^{r_0}r^{2m-1}\Lambda(r)^2\,dr=\infty.
\]
For $\Lambda(r)=r^{-\alpha}(\log(1/r))^\beta$, this holds if $m<\alpha$, or if $m=\alpha$ and $\beta\ge-1/2$. This is a sufficient criterion for the cutoff argument, not a claimed asymptotic classification of boundary ends.
\end{remark}

Next, we present a construction of mixed cutoff functions.

\begin{lemma}\label{lem:mixed-cutoffs}
Assume that every selected boundary end satisfies \eqref{eq:capacity-condition} for $m=n_j$. Then there is a sequence $0\le\Phi_\ell\le1$ of compactly supported Lipschitz functions, eventually equal to one on every compact subset of $\Sigma$, such that
\begin{equation}\label{eq:mixed-error}
\int_\Sigma|\nabla\Phi_\ell|^2|\omega|^2\,dA\longrightarrow0
\qquad\text{for every }\omega\in\mathcal V.
\end{equation}
No metric growth hypothesis is needed at the unselected boundary ends.
\end{lemma}
\begin{proof}
Call the interior punctures and the selected boundary punctures
\emph{active}. Choose pairwise disjoint coordinate neighborhoods of
these points. In each neighborhood, Lemmas~\ref{lem:interior-asymptotics}
and~\ref{lem:radial-capacity} provide cutoffs taking values in $[0,1]$,
vanishing near the puncture, and equal to one outside a neighborhood
shrinking to that puncture. These cutoffs depend only on the metric and
the allowed pole order, so the same sequence works for every
$\omega\in\mathcal V$.

Define $\zeta_\ell$ by these local cutoffs in the chosen neighborhoods
and by one elsewhere. The definitions agree near the boundaries
of the coordinate neighborhoods, so $\zeta_\ell$ is Lipschitz.
Its gradient is supported in the finitely many transition annuli.
Summing the local estimates therefore gives
\[
\int_\Sigma|\nabla\zeta_\ell|^2|\omega|^2\,dA
\longrightarrow0,
\qquad\text{for every }\omega\in\mathcal V.
\]
Moreover, $\zeta_\ell$ is eventually one on every compact subset
of $\Sigma$. If there are no active punctures, simply take
$\zeta_\ell\equiv1$.
The unselected boundary ends may prevent $\zeta_\ell$ from having
compact support. Nevertheless,
\[
\zeta_\ell\omega\in L^2(\Sigma)
\qquad\text{for each fixed }\ell.
\]
Indeed, this form vanishes near every possible pole. At each
remaining puncture, $\omega$ extends smoothly in the compactification
and is therefore locally square-integrable in a smooth compactified
metric. Conformal invariance gives the same integrability in the
induced metric. The region left after removing these puncture
neighborhoods is compact, where integrability is automatic.

Suppose first that $h=\dim\mathcal V>0$, and choose a real basis
$\omega_1,\ldots,\omega_h$. Let $\theta_R$ be the intrinsic cutoff
from \eqref{eq:cutoffs}. For each $\ell$, the number
\(
M_\ell=\sum_{a=1}^h
\|\zeta_\ell\omega_a\|_{L^2}^2
\)
is finite. We can therefore choose $R_\ell\ge\ell$ such that
\(
\frac{C^2M_\ell}{R_\ell^2}\le\ell^{-2}.
\)
Set
\[
\Phi_\ell=\theta_{R_\ell}\zeta_\ell.
\]
Then $0\le\Phi_\ell\le1$, and $\Phi_\ell$ is Lipschitz with compact
support contained in $\overline B_{2R_\ell}$.
Using the product rule, $0\le\theta_{R_\ell},\zeta_\ell\le1$,
and $|\nabla\theta_{R_\ell}|\le C/R_\ell$, we obtain
\begin{align*}
\int_\Sigma|\nabla\Phi_\ell|^2|\omega_a|^2\,dA
&\le
2\int_\Sigma|\nabla\zeta_\ell|^2|\omega_a|^2\,dA
+2\int_\Sigma
|\nabla\theta_{R_\ell}|^2|\zeta_\ell\omega_a|^2\,dA\\
&\le
2\int_\Sigma|\nabla\zeta_\ell|^2|\omega_a|^2\,dA
+\frac{2C^2}{R_\ell^2}
\|\zeta_\ell\omega_a\|_{L^2}^2\\
&\le
2\int_\Sigma|\nabla\zeta_\ell|^2|\omega_a|^2\,dA
+2\ell^{-2}
\longrightarrow0.
\end{align*}
For any fixed $\omega=\sum_{a=1}^h c_a\omega_a$,
the pointwise inequality
\[
|\omega|^2
\le
\left(\sum_{a=1}^h c_a^2\right)
\left(\sum_{a=1}^h|\omega_a|^2\right)
\]
then proves the required convergence for $\omega$.

Finally, on any fixed compact subset, $\zeta_\ell$ is eventually
one, and so is $\theta_{R_\ell}$ because $R_\ell\to\infty$.
Thus $\Phi_\ell$ is eventually one there.
If $\mathcal V=\{0\}$, the intrinsic cutoffs alone suffice.
The radii $R_\ell$ are chosen after $\zeta_\ell$, so no uniform
bound on $M_\ell$ is needed.
\end{proof}

\subsection{Comparison with the logarithmic weight}
The admissible interior orders agree with the weighted construction of \cite[Section~3]{CM23}. Their weight is
\[
W(F)=\frac{1}{(1+|F|^2)\log^2(2+|F|)}.
\]
At an interior end, $|F(z)|\asymp r^{-d}$, so conformal invariance gives, for a local representative $dz/z^\ell$,
\[
\int_E W(F)|\omega|^2\,dA
\asymp\int_0^{r_0}
\frac{r^{2d-2\ell+1}}{\log^2(1/r)}\,dr.
\]
At the critical order $\ell=d+1$, the integral is
$\int_0^{r_0}dr/(r\log^2(1/r))<\infty$, and thus the logarithmic factor is decisive. A cutoff transitioning over $R<|F|<R^2$ satisfies
\[
|\nabla\varphi_R|^2\le\frac{C}{|F|^2\log^2R}\le C'W(F)
\]
on that annulus, so weighted integrability controls its error by a tail integral. In this comparison, the norm of a complex differential may equivalently be replaced by the sum of the squared norms of its real and imaginary parts, and the harmless constant has no effect on integrability.

Our construction instead uses bounded geometric norms at active power-law ends and intrinsic cuts at regular punctures. It does not assert global properness of $F$ or weighted integrability at unselected boundary ends. The local finite-curvature description is available at interior ends; a boundary puncture alone does not provide its analogue. This is why Theorem~C states a separate condition only at the boundary ends whose poles are used.

\subsection{Index reduction and the exceptional space}
Let $\KK(\mathcal V)$ denote the forms in $\mathcal V$ whose three Euclidean components solve $Ju_i=0$ and $Bu_i=0$. These components need not lie in unweighted $L^2$.

\begin{proposition}\label{prop:meromorphic-index}
Under the hypotheses of Lemma~\ref{lem:mixed-cutoffs},
\[
\dim\mathcal V\le3I+\dim\KK(\mathcal V).
\]
Moreover, $\dim\KK(\mathcal V)\le1$. It is zero if an unselected boundary puncture exists, or if $|\omega|_\gamma\to0$ along a boundary arc approaching a puncture for every $\omega\in\mathcal V$.
\end{proposition}
\begin{proof}
Let $\EE=\operatorname{span}\{\psi_1,\ldots,\psi_I\}$ be the negative
space from Lemma~\ref{lem:negative-space}. For
$\omega\in\mathcal V$, put $X=\omega^\sharp$ and
$u_i=\langle X,e_i\rangle$, and define
\[
\mathcal F:\mathcal V\longrightarrow\R^{3I},
\qquad
\mathcal F(\omega)=\bigl(Q(u_i,\psi_j)\bigr)_{i,j}.
\]
These pairings are well defined because the $\psi_j$ have compact
support.

Suppose that $\mathcal F(\omega)=0$. Write
$\omega_\ell=\Phi_\ell\omega$, where $\Phi_\ell$ is furnished by
Lemma~\ref{lem:mixed-cutoffs}. For all sufficiently large $\ell$,
$\Phi_\ell=1$ on a neighborhood of $\supp\EE$. By locality of $Q$,
\[
Q(\Phi_\ell u_i,\psi_j)=Q(u_i,\psi_j)=0
\]
for every $i,j$; thus $\Phi_\ell u_i\in\EE^{\perp_Q}$. The
nonnegativity of $Q$ on this space and the localized identity yield \eqref{eq:localized}
\[
0\le Q(\Phi_\ell u_i,\Phi_\ell u_i)
\le \sum_{a=1}^3 Q(\Phi_\ell u_a,\Phi_\ell u_a)
\le \int_\Sigma |\nabla\Phi_\ell|^2|\omega|^2\,dA
\longrightarrow 0.
\]

Let $v$ be an arbitrary compactly supported smooth test function.
By Lemma~\ref{lem:negative-space}, write
$v=v^\perp+e$, with $v^\perp\in\EE^{\perp_Q}$ and $e\in\EE$.
Cauchy--Schwarz for the nonnegative form $Q$ on $\EE^{\perp_Q}$
gives
\[
|Q(\Phi_\ell u_i,v^\perp)|^2
\le Q(\Phi_\ell u_i,\Phi_\ell u_i)\,Q(v^\perp,v^\perp)
\longrightarrow 0,
\]
while $Q(\Phi_\ell u_i,e)=0$. Since $\Phi_\ell=1$ near
$\supp v$ for large $\ell$, locality also gives
$Q(u_i,v)=Q(\Phi_\ell u_i,v)$. Hence $Q(u_i,v)=0$ for every such
$v$. Interior tests imply $Ju_i=0$, and integration by parts
against tests with arbitrary compactly supported boundary traces
implies $Bu_i=0$. Hence,
$\ker\mathcal F\subset\KK(\mathcal V)$, and rank--nullity proves
\[
\dim\mathcal V\le 3I+\dim\KK(\mathcal V).
\]

It remains to bound the exceptional space. The argument of
Lemma~\ref{lem:constant} is local and uses no global $L^2$
assumption. Thus, for every $\omega\in\KK(\mathcal V)$ and each
connected component $\Gamma$ of the geometric boundary,
\[
X=c_\Gamma T\quad\text{on }\Gamma
\]
for some constant $c_\Gamma$. For any fixed component $\Gamma$,
the map $\omega\mapsto c_\Gamma$ is injective. Indeed, if
$c_\Gamma=0$, then $\omega$ vanishes on an open boundary arc,
and boundary uniqueness forces $\omega=0$. In particular,
$\dim\KK(\mathcal V)\le1$.

Suppose first that $q_j$ is an unselected boundary puncture.
Every $\omega\in\mathcal V$ extends smoothly across $q_j$.
Choose an adjacent boundary arc $\alpha$ approaching $q_j$,
parametrized by a smooth coordinate of the compactification.
Smooth extension gives
\[
\int_\alpha|\omega(\alpha')|\,dt<\infty.
\]
On the other hand, completeness makes the induced length of
$\alpha$ infinite. Since $X=c_\Gamma T$ on the component
containing $\alpha$,
\[
|\omega(\alpha')|=|c_\Gamma|\,|\alpha'|_\gamma.
\]
It follows that $c_\Gamma=0$, and boundary uniqueness gives
$\omega=0$.

Alternatively, suppose that $|\omega|_\gamma\to0$ along a
boundary arc approaching a puncture for every
$\omega\in\mathcal V$. For an exceptional form,
$|\omega|_\gamma=|c_\Gamma|$ along that arc. Thus
$c_\Gamma=0$, and boundary uniqueness again gives $\omega=0$.
In either case, $\KK(\mathcal V)=\{0\}$.
\end{proof}

\begin{proof}[Proof of Theorems B and C]
For Theorem~C, consider the space $\mathcal V$ with poles of order
at most $d_i+1$ at each interior puncture, poles of order at most
$n_j$ at each selected boundary puncture, and removable
singularities at the remaining boundary punctures.
At a selected boundary end, the assumed lower bound gives
\[
a_{n_j}(r)
=r^{1-2n_j}\int_0^\pi e^{-2\lambda_j(r,\theta)}\,d\theta
\le \pi C_j^2r.
\]
Hence $\int_0^{r_0}a_{n_j}(r)^{-1}\,dr=\infty$.
The corresponding condition at each interior end follows from
Lemma~\ref{lem:interior-asymptotics}.
Thus Lemma~\ref{lem:mixed-cutoffs} provides a common sequence
of compactly supported cutoffs with vanishing error for
every form in $\mathcal V$.

If $L>0$, Lemma~\ref{lem:meromorphic-dimension} and
Proposition~\ref{prop:meromorphic-index} yield
\[
G-1+L=\dim_{\R}\mathcal V
\le 3I+\dim\KK(\mathcal V).
\]
When some boundary puncture is unselected, all forms in $\mathcal V$
extend smoothly there, so $\KK(\mathcal V)=\{0\}$.
When every boundary puncture is selected, the general estimate
$\dim\KK(\mathcal V)\le1$ applies. By the definition of
$\varepsilon_S$, we therefore obtain
\[
3I\ge G-1+L-\varepsilon_S.
\]
Together with $3I\ge G$ from Theorem~A, this gives
\[
3I\ge\max\{G,\ G-1+L-\varepsilon_S\},
\]
which is \eqref{eq:mixed-main} because $G=2g+b-1$.

Suppose now that, at some selected boundary puncture,
$e^{\lambda_j}\ge C_j^{-1}|z|^{-\alpha_j}$ with $\alpha_j>n_j$.
For $\omega=\Rea(f\,dz)\in\mathcal V$, the pole-order restriction
implies $|f(z)|\le C_\omega|z|^{-n_j}$, and hence
\[
|\omega|_\gamma=e^{-\lambda_j}|f|
\le C_jC_\omega|z|^{\alpha_j-n_j}\longrightarrow0.
\]
For an exceptional form, its dual vector field equals
$c_\Gamma T$ on the boundary component containing an arc
approaching this puncture. Its norm is therefore constant
along that arc, and the preceding decay forces $c_\Gamma=0$.
Boundary uniqueness gives $\omega=0$. Thus
$\KK(\mathcal V)=\{0\}$, and $\varepsilon_S$ may be replaced
by zero even when every boundary puncture is selected.

If $L=0$, there are no interior punctures and no selected
boundary punctures, since every allowed pole order is positive.
All forms in $\mathcal V$ therefore extend smoothly across
the punctures. Lemma~\ref{lem:meromorphic-dimension} gives
$\dim_{\R}\mathcal V=G$, and the existence of a boundary puncture
eliminates the exceptional space, recovering Theorem~A.

For Theorem~B, take $S=\varnothing$. No additional boundary
growth assumption is then required, and every form in
$\mathcal V$ is regular at all boundary punctures.
Since $s\ge1$, Proposition~\ref{prop:meromorphic-index} gives
$\KK(\mathcal V)=\{0\}$. If $r\ge1$, then
\[
L=2\sum_{i=1}^r(d_i+1)>0,
\]
so
\[
3I\ge G-1+L
=2g+b-2+2\sum_{i=1}^r(d_i+1).
\]
Taking the ceiling after division by three proves
\eqref{eq:interior-main}. Finally, embedded interior ends
have $d_i=1$ by Lemma~\ref{lem:interior-asymptotics}, giving
$3I\ge2g+b+4r-2$.
\end{proof}

The weak-index versions follow by replacing each strong-index lower bound $I\ge M$ by $\Ind_w\ge M-1$. In particular, no separate global weighted spectral theory is needed for these extensions; only pairings against a fixed compactly supported negative space are used.

\begin{remark}
Theorem~C remains valid with \eqref{eq:capacity-condition} in place of \eqref{eq:intro-growth}. Ends not selected contribute no pole order to $L$. Selecting all ends is not obligatory, and a regular boundary puncture can be more useful than a marginal additional pole because it eliminates the exceptional space. None of these estimates asserts vanishing of the full Jacobi nullity.
\end{remark}

\section{Examples and scope of the estimate}\label{sec:scope}
\begin{example}
In the Euclidean half-space $\Omega=\{x_3>0\}$, the surface
$\Sigma=\{(x_1,0,x_3):x_3\ge0\}$ is a complete free boundary minimal half-plane. Both $A$ and $q$ vanish, so $Q(f,f)=\int_\Sigma|\nabla f|^2\dd A$ and its strong index is zero. Its compactification is a disk with one boundary point removed, so $g=0$, $b=1$, and the harmonic-form dimension is zero. This example shows that the assumptions allow $H_{\partial\Omega}$ to vanish identically.
\end{example}

\begin{example}
Let $\Omega=\{0<x_3<1\}$ and
$\Sigma=\{(x_1,0,x_3):0\le x_3\le1\}$. This is complete, free boundary, minimal, and stable. The conformal map $z\mapsto e^{\pi z}$ sends the interior of the strip $0<\operatorname{Im}z<1$ to the upper half-plane; its two horizontal boundary lines map to the two intervals of the real axis separated by zero. A Cayley transform identifies the compactification with a closed disk with two boundary points removed. Thus $g=0$ and $b=1$, although $\partial\Sigma$ has two connected components. The dimension in \eqref{eq:hodge-dimension} is again zero.
\end{example}

\begin{example}\label{ex:strip-poles}
For the strip of the preceding example, the coordinate $w=e^{\pi z}$ gives
\[
\gamma=\frac{|dw|^2}{\pi^2|w|^2}
\]
near the boundary puncture $w=0$; the same expression holds near the other end using $1/w$. Thus $e^\lambda\asymp|w|^{-1}$. Simple poles satisfy Lemma~\ref{lem:radial-capacity}, but double poles do not satisfy its radial criterion. Indeed, for a double pole $a_2(r)\asymp r^{-1}$, so $\int_0^{r_0}a_2(r)^{-1}\,dr<\infty$. In particular, embeddedness alone does not imply \eqref{eq:intro-growth} with $n_j=2$. Assigning order two to both boundary ends in \eqref{eq:mixed-main} would incorrectly give $3I\ge1$, whereas the strip is stable. Allowing simple poles at both ends gives $L=2$, $G=0$, and $G-1+L-1=0$, consistently with stability. The form $dx_1$ spans the exceptional space in this example, so the possible one-dimensional loss is real.
\end{example}

\begin{example}\label{ex:fractional}
Fix $0<a\le1$ and define
\[
F_a(z)=(z+i)^a,\qquad \operatorname{Im}z\ge0,
\]
using the branch with $0<\arg(z+i)<\pi$. This is a conformal embedding of the closed upper half-plane onto a closed planar domain $D_a$. It is proper, since $|F_a(z)|=|z+i|^a\to\infty$ as $|z|\to\infty$. Its boundary is the smooth proper curve $\beta_a(t)=(t+i)^a$, and
\[
\frac{d}{dt}\arg\beta_a'(t)=\frac{1-a}{1+t^2},
\qquad
k_a(t)=\frac{1-a}{a}(1+t^2)^{-(a+1)/2}\ge0.
\]
To verify convexity globally, rotate the curve so that its tangent angles belong to the interval $[-(1-a)\pi/2,(1-a)\pi/2]$. Its first coordinate is strictly increasing and its tangent angle is nondecreasing; it is therefore a convex graph, with $D_a$ on its convex side. For $a=1$ the domain is a half-plane.

Set $\Omega_a=\operatorname{int}(D_a)\times\R$ and $\Sigma_a=D_a\times\{0\}$. The surface is complete because $D_a$ is closed and convex, and it meets the smooth mean-convex cylindrical support orthogonally. Its second fundamental form is zero and $q=0$, since its normal is the vertical cylinder direction. Thus $I=0$. Moreover,
\[
\int_{\partial\Sigma_a}\mathscr H\,ds
=\int_{-\infty}^{\infty}\frac{1-a}{1+t^2}\,dt
=(1-a)\pi.
\]
There is one boundary puncture, at $z=\infty$. With $\zeta=-1/z$,
\[
\gamma=a^2|\zeta|^{-2(a+1)}|1-i\zeta|^{2(a-1)}|d\zeta|^2.
\]
Hence the boundary exponent is $a+1\in(1,2]$, despite proper embeddedness, zero Gaussian curvature, and finite total support mean curvature along the boundary. For $a<1$, the sufficient power criterion admits order one but not order two. This example rules out an automatic integer-multiplicity description of general boundary ends; it does not rule out universal admissibility of simple poles.
\end{example}

\begin{proposition}\label{prop:reflection}
Suppose a boundary end, outside a compact truncation, is properly embedded, lies on one side of a plane $\Pi$, and has its entire geometric boundary in $\Pi$, meeting it orthogonally. Then the induced metric near its boundary puncture satisfies
\[
e^\lambda\asymp|z|^{-2}.
\]
Such an end can be selected with $n_j=2$ in Theorem~C.
\end{proposition}
\begin{proof}
Reflection in $\Pi$ extends the end to a smooth minimal end without geometric boundary. The two interiors lie in opposite open half-spaces, so the reflected end is embedded. It is complete towards its puncture and has finite total curvature, since reflection doubles the curvature integral of the original end. The double of its punctured half-disk parameter domain is a punctured disk. The classical embedded-end description, as in Lemma~\ref{lem:interior-asymptotics}, gives multiplicity one and metric factor comparable to $|z|^{-2}$. Restriction to one half proves the assertion. Reflection of complete free boundary surfaces in a half-space is discussed in \cite[Section~2]{Chen21}; only this local end argument is needed here.
\end{proof}

If $s_{\mathrm{ref}}$ ends satisfy Proposition~\ref{prop:reflection}, they contribute $2s_{\mathrm{ref}}$ to $L$. No planar condition is imposed on unselected ends. Both margins of a selected end must lie in the same plane: the two parallel supporting planes of a strip do not meet this hypothesis. An asymptotically planar support, without quantitative control, is not covered by the proposition.

More generally, if a boundary circle of $\Sigmabar$ contains $m\ge1$ punctures, its complement has $m$ connected components, whereas a boundary circle with no punctures remains one component. This explains why replacing $b$ by the number of components of $\partial\Sigma$ would give an incorrect topological interpretation; the stable strip already rules out that replacement in \eqref{eq:main}.

For strong index $I$, Theorem~A yields $2g+b\le3I+1$. In particular, strong stability forces $g=0$ and $b=1$. This is a statement about the compactification, not a geometric classification: Theorem~A alone does not determine the number of punctures. Theorems~B and~C provide additional restrictions when interior ends or admissible boundary poles are present. Likewise, it does not classify equality or establish that the factor $1/3$ is optimal.

\section*{Funding}
The first author was partially supported by the Brazilian National Council for Scientific and Technological Development (CNPq) [Grants 402563/2023-9 and 304381/2026-8 to M.B.], and was supported by the Coordination for the Improvement of Higher Education Personnel (CAPES), Finance Code 001.

\section*{Data availability statement}
This manuscript has no associated data.
\section*{Conflict of interest statement}
The authors declare that they have no conflict of interest.

\section*{AI disclosure}
The authors used a generative AI tool to assist with language editing and to discuss the presentation of some arguments. The authors independently checked the mathematical statements, proofs, and references and take full responsibility for the content of this article.

\end{document}